\documentclass[a4paper,10pt]{article}
\usepackage[francais]{babel}
\usepackage[latin1]{inputenc}
\usepackage[T1]{fontenc}
\usepackage[sans]{dsfont}
\usepackage{mathrsfs}
\usepackage{amssymb, amsmath, amsthm}
\usepackage[all]{xy}
\usepackage{multicol}
\usepackage{fullpage}

\makeatletter
\def\section{\@ifstar\unnumberedsection\numberedsection}
\def\numberedsection{\@ifnextchar[
  \numberedsectionwithtwoarguments\numberedsectionwithoneargument}
\def\unnumberedsection{\@ifnextchar[
  \unnumberedsectionwithtwoarguments\unnumberedsectionwithoneargument}
\def\numberedsectionwithoneargument#1{\numberedsectionwithtwoarguments[#1]{#1}}
\def\unnumberedsectionwithoneargument#1{\unnumberedsectionwithtwoarguments[#1]{#1}}
\def\numberedsectionwithtwoarguments[#1]#2{%
  \ifhmode\par\fi
  \removelastskip
  \vskip 5ex\goodbreak
  \refstepcounter{section}%
  \hbox to \hsize{\hss\vbox{\advance\hsize by 1cm
      \noindent
      \leavevmode\huge\bfseries\raggedright
      \thesection.\ 
      #2\par
      \vskip -2ex
      \noindent\hrulefill
      }}\nobreak
  \vskip 2ex\nobreak
  }
\def\unnumberedsectionwithtwoarguments[#1]#2{%
  \ifhmode\par\fi
  \removelastskip
  \vskip 5ex\goodbreak
  \hbox to \hsize{\hss\vbox{\advance\hsize by 1cm
      \noindent
      \leavevmode\huge\bfseries\raggedright
      #2\par
      \vskip -2ex
      \noindent\hrulefill
      }}\nobreak
  \vskip 2ex\nobreak
  }
\makeatother

\newtheoremstyle{THEOREME}{}{}{\sffamily}{}{\bfseries\scshape}{.}{.5em}{}
\theoremstyle{THEOREME}
\newtheorem{theoreme}{Th\'eor\`eme}[subsection]
\newtheorem{corollaire}[theoreme]{Corollaire}
\newtheorem{proposition}[theoreme]{Proposition}
\newtheorem{notation}[theoreme]{Notation}

\newtheoremstyle{DEFINITION}{0.5cm}{0.5cm}{\itshape}{}{\bfseries}{.}{.5em}{}
\theoremstyle{DEFINITION}
\newtheorem{definition}[theoreme]{D\'efinition}

\newtheoremstyle{REMARQUE}{}{}{\normalfont}{}{\bfseries}{.}{.5em}{}
\theoremstyle{REMARQUE}
\newtheorem{remarque}[theoreme]{Remarque}

\newtheoremstyle{LEMME}{}{}{\slshape}{}{\bfseries}{.}{.5em}{}
\theoremstyle{LEMME}
\newtheorem{lemme}[theoreme]{Lemme}

\input cyracc.def
\font\tencyr=wncyr6
\def\cyr{\tencyr\cyracc}
\newcommand{\cyrb}{{\cyr B}}

\newcommand{\SW}[1]{\textbf{\textit{\emph{\textsf{#1}}}}}
\newcommand{\SWm}[1]{{\textit{\emph{\textsf{#1}}}}}
\newcommand{\dpl}[1]{\displaystyle{#1}}

\def\ie{\textit{i.e. }}
\def\cf{\textit{c.f. }}

\newcommand{\Ob}{\SWm{Ob}}

\newcommand{\A}{\mathscr{A}}

\newcommand{\CC}{\mathscr{C}}

\newcommand{\T}{\mathscr{T}}

\renewcommand{\H}{\mathscr{H}}
\newcommand{\G}{\mathscr{G}}
\newcommand{\comp}{{\SWm{C}}}
\newcommand{\qhtt}{{\SWm{QK}}}
\newcommand{\htt}{{\SWm{K}}}
\newcommand{\Sch}{{\mathfrak{Sch}}}
\newcommand{\opp}{\textsf{opp}}
\newcommand{\Hom}{\SWm{Hom}}
\newcommand{\HHom}{\SWm{\underline{Hom}}}
\newcommand{\Id}{\SWm{Id}}

\newcommand{\CCp}[1]{\CC_{{w}\geqslant {#1}}}
\newcommand{\CCn}[1]{\CC_{{w}\leqslant {#1}}}
\newcommand{\CCe}[1]{\CC_{{w}= {#1}}}
\newcommand{\obp}[2]{{#1}_{w \geqslant #2}}
\newcommand{\obn}[2]{{#1}_{w \leqslant #2}}
\newcommand{\Xp}[1]{\obp{X}{#1}}
\newcommand{\Xn}[1]{\obn{X}{#1}}

\newcommand{\W}{\SWm{W}}
\newcommand{\Wc}{{\check{\SWm{W}}}}
\newcommand{\chic}{{\check{\chi}}}

\newcommand{\Chow}[1]{{\SWm{Chow}(#1)}}
\newcommand{\DMB}[1]{{\SWm{DM}_{\text{\cyrb}}(#1)}}
\newcommand{\DMBc}[1]{{\SWm{DM}_{\text{\cyrb},c}(#1)}}

\newcommand{\DMBe}[1]{{{\SWm{DM}_{\text{\cyrb}}(#1)}_{W= 0}}}
\newcommand{\DMBcp}[1]{{{\SWm{DM}_{\text{\cyrb},c}(#1)}_{w\geqslant 0}}}
\newcommand{\DMBcn}[1]{{{\SWm{DM}_{\text{\cyrb},c}(#1)}_{w\leqslant 0}}}
\newcommand{\DMBce}[1]{{{\SWm{DM}_{\text{\cyrb},c}(#1)}_{w= 0}}}
\newcommand{\tate}{{\SW{Tate}}}
\newcommand{\DMBcT}[1]{{\SWm{DM}_{\text{\cyrb},c,T}(#1)}}
\newcommand{\DMBcTe}[1]{{{\SWm{DM}_{\text{\cyrb},c,T}(#1)}_{\tate= 0}}}
\newcommand{\Pos}{\SWm{POS}}

\newcommand{\HB}{\SWm{H}_{\text{\cyrb}}}

\newcommand{\N}{\mathbb{N}}
\newcommand{\Z}{\mathbb{Z}}
\newcommand{\Q}{\mathbb{Q}}

\newcommand{\C}{\mathbb{C}}

\newcommand{\morph}[2]{#1 \rightarrow #2}
\newcommand{\morphp}[3]{#1 \rightarrow #2 \rightarrow #3}
\newcommand{\immouv}[1][r]{\ar@{^(->}[#1]|*[@]{\circ}\ar[#1]}
\newcommand{\immfer}[1][r]{\ar@{^(->}[#1]|*[@]{\shortmid}\ar[#1]}
\newcommand{\immop}[2]{\raisebox{.7ex}{\xymatrix{#1 \immouv[r]& #2}}}
\newcommand{\immcl}[2]{\raisebox{.7ex}{\xymatrix{#1 \immfer[r]& #2}}}
\newcommand{\tridis}[3]{#1 \rightarrow #2 \rightarrow #3 \overset{+1}{\longrightarrow}}
\def\cartesien{\ar@{}[rd]|{\square}}

\newcommand{\Kar}{\mathfrak{R}}
\newcommand{\SE}{{\mathfrak{Ext}}}
\newcommand{\CA}[1]{{\left\langle{#1}\right\rangle}}
\newcommand{\CAep}[1]{{\left\langle{#1}\right\rangle}^{\text{ép}}}
\newcommand{\CAinf}[1]{{\left\langle{#1}\right\rangle}_{\infty}}
\newcommand{\CAepinf}[1]{{\left\langle{#1}\right\rangle}^{\text{ép}}_{\infty}}

\newcommand{\Un}{\mathds{1}}
\newcommand{\D}{\SWm{D}}
\newcommand{\KK}{\SWm{K}}
\newcommand{\Spec}{\SWm{Spec}}
\newcommand{\Dim}{\SWm{Dim}}
\newcommand{\overlined}[1]{\overline{\overline{#1}}}

\newenvironment{quot}{
\begin{quote}
$\ulcorner$ \scriptsize
}
{
\normalsize
$\hfill\lrcorner$
\end{quote}
}

\title{Complexe de Poids,\\
Dualité\\
et\\
Motifs de Beilinson}
\author{David H\'ebert}
\date{}

\begin{document}

\maketitle

\renewcommand{\abstractname}{\textsc{Abstract}}
\begin{abstract}
\noindent In the article \cite{GS95}, Gillet and Soulé define a weight complex on the category of Voevodsky motives over a field of characteristic $0$. In \cite{Bo}, Bondarko generalizes this construction for any $f$-category with a bounded weight structure, as is the case for Beilinson motives (following Cisinski-Déglise ; \cite{CD}). The first purpose of this note is to generalize \cite[thm. 2]{GS95} in the world of Beilinson motives. This done, we will naturally be led to define the motivic Euler characteristic dual to that considered by Bondarko in \cite{Bo2}. This fact will motivate the second line of this note : proving that the duality operation exchanges the weight as is the case for $t$-structure (\cite[5.1.14.$(iii)$]{BBD}).
\end{abstract}

\tableofcontents

\section*{Introduction}\markright{INTRODUCTION}
\addcontentsline{toc}{section}{Introduction}
Dans l'article \cite{GS95}, Gillet et Soulé définissent un complexe de poids sur la catégorie des motifs de Voevodsky définie sur un corps de caractéristique $0$. Dans \cite{Bo}, Bondarko généralise cette construction pour toute $f$-catégorie munie d'une structure de poids bornée, comme c'est le cas pour les motifs de Beilinson (suivant Cisinski-Déglise ; \cite{CD}). Le premier but de cette note est de généraliser \cite[thm. 2]{GS95} dans le monde des motifs de Beilinson. Ceci fait, on sera naturellement amené à définir la caractéristique d'Euler motivique duale à celle considérée par Bondarko dans \cite{Bo2}. Ce fait motivera le second axe de cette note : prouver que l'opération de dualité échange les poids comme c'est le cas pour les $t$-structures (\cite[5.1.14.$(iii)$]{BBD}).

La première partie regroupe les résultats de la théorie de Bondarko que nous utiliserons ici. On applique ces constructions, dans la seconde partie, aux motifs de Beilinson : on commence par faire quelques rappels sur cette catégorie de motifs (section \ref{sec.rap.mot}), puis on  prouve que la dualité renverse les poids (corollaire \ref{prop.dual.conj}). On donne ensuite un encadrement du poids du motif associé à un schéma (théorème \ref{thm.mesure.motifs}) et des résultats d'annulation cohomologique (corollaire \ref{cor.ann.coh}) que l'on applique à la catégorie des motifs de Tate (\ref{cor.struct.Poids.Tate}). Finalement, on établit la version relative de \cite[thm. 2. $(i)$, $(ii)$, $(iii)$, $(iv)$]{GS95} (theorème \ref{equi.GS95.thm2}) que l'on utilise pour définir la caractéristique d'Euler motivique (corollaire \ref{cor.car.Euler}).

\noindent \textbf{Notations et conventions.}

\noindent Si $\CC$ est une catégorie, la notation $\H\subset \CC$ (où $\CC\supset\H$) signifiera toujours que $\H$ est une sous-catégorie pleine de $\CC$. Pour cette raison nous décrirons les sous-catégories pleines uniquement par la classe de leurs objets. Nous adopterons également les notations ensemblistes ($\in$, $\exists$, $\cup$, $\cap$, etc.) pour les catégories. Par exemple, la notation $X\in \CC$ signifiera toujours que $X$ est un objet de $\CC$. Les triangles distingués seront notés $\tridis{A}{B}{C}$. On dira qu'un foncteur entre catégories triangulées est un foncteur de catégories triangulées s'il s'agit d'un foncteur additif transformant tous triangles distingués en triangles distingués.

Tous les schémas considérés sont de type fini sur une base $B$ excellente de dimension de Krull au plus $2$. 
Les morphismes entre schémas sont séparés. On note $\Sch/S$ la catégorie des schémas séparés au dessus de $S$.

On reprend une large partie des notations introduites dans \cite[déf. 1.1, 1.3, 1.11]{Heb} ; entre autre : si $\H\subset\CC$, $\Kar(\H)$ désigne la sous-catégorie pleine de $\CC$ formée des rétractes d'objets de $\H$ et ${}^{\bot}\H$ (resp. $\H^\bot$) la sous-catégorie des objets orthogonaux à gauche (resp. à droite) de $\H$. Lorsque $\CC$ est triangulée et $n\in \N$, $\SE^n_\CC{\H}$ (resp. $\SE_\CC(\H)$) désigne les $n$-extensions (resp. les extensions) d'objets de $\H$. La catégorie $\CA{\H}$ (resp. $\CAep{\H}$) désigne la sous-catégorie triangulée (resp. la sous-catégorie triangulée et épaisse) engendrée par $\H$. On peut ajouter à ces constructions des petites sommes (lorsqu'elles existent dans $\CC$), $\H^\infty$, $\SE_\CC^\infty{\H}$, $\CAinf{\H}$, $\CAepinf{\H}$.

\newpage

\section{Outils et sorites.}
\addcontentsline{toc}{section}{\protect\numberline{\thesection} Outils et sorites.}

\subsection{Quasi-homotopies et homotopies.}

Soit $\A$ une cat\'egorie poss\'edant un objet \`a la fois final et
initial, not\'e $0$.

\begin{definition}
Un \SW{complexe} $A$ de $\A$ est la donn\'ee d'une famille $(A^n)_{n
\in \Z}$ d'objets de $\A$ et d'une famille de morphismes $(d_A^n)_{n
\in \Z}$, appel\'es \SW{diff\'erentielles}, telles que pour tout $n
\in \Z$, $d_A^n \in \Hom_\A(A^n,A^{n+1})$ et $d_A^{n+1}d_A^{n}=0$.
\end{definition}

\begin{definition}
Soient $A$ et $B$ deux complexes de $\A$. Un \SW{morphisme de
complexes} $f:\morph{A}{B}$ est une suite de morphismes $(f^n)_{n
\in \Z}$ o\`u $\dpl{f^n: \morph{A^n}{B^n}}$ telle que  pour tout $n
\in \Z$, $d_B^nf^n=f^{n+1}d_A^n$.
\end{definition}

On consid\`ere alors la cat\'egorie des complexes de $\A$ not\'ee
$\comp(\A)$.

\begin{definition}\label{comp.tronc}
Un complexe $A$ de $\A$ est dit \SW{born\'e inf\'erieurement} (resp.
\SW{born\'e sup\'erieurement}) s'il existe $N \in \Z$ tel que
$$\forall n \in \Z , \; n<N \quad A^n = 0, \qquad
(\text{resp.}\quad  \forall n \in \Z , \; n>N \quad A^n = 0)$$ Un
complexe sera dit \SW{born\'e}, s'il est \`a la fois born\'e
inf\'erieurement et sup\'erieurement. On note $\comp^b(\A)$ (resp.
$\comp^+(\A)$, resp. $\comp^-(\A)$), la sous-cat\'egorie pleine de
$\comp(\A)$ dont les objets sont les complexes born\'es (resp.
born\'es inf\'erieurement, resp. born\'es sup\'erieurement).

Soit $P$ une partie de $\Z$. On note
$\comp^{P}(\A)$, la sous-cat\'egorie pleine de $\comp(\A)$
form\'ee des complexes $A$ telle que pour tout entier
$k\notin P$, $A^k=0$.
\end{definition}

On d\'esigne par $\comp^*(\A)$ l'une de ces cat\'egories ; le
symbole $*$ repr\'esentant $\{+, -, b\}$ ou une partie de $\Z$.

\begin{remarque}
Si $\A$ est une cat\'egorie additive (resp. ab\'elienne), alors il
en va de m\^eme pour $\comp^*(\A)$.
\end{remarque}

On fixe une cat\'egorie additive $\A$, et trois objets  $A$, $B$ et
$C$ de $\comp^*(\A)$.

\begin{definition}\label{def.homot.rela}
Soient $f$ et $g$ deux \'el\'ements de $\Hom_{\comp^*(\A)}(A,B)$. On
dit que $f$ et $g$ sont \SW{quasi-homotopes} (resp. \SW{homotopes}), s'il existe une collection
de morphismes $\{s^n\}_{n \in \Z}$ et $\{t^n\}_{n\in \Z}$ telle que
pour tout $k \in \Z$, $s^k$ et $t^k$ sont des \'el\'ements de
$\Hom_{\A}(A^k,B^{k-1})$ et telle que
$$f^k-g^k = s^{k+1}d_A^k + d_B^{k-1}t^k
\qquad(\text{resp.}\quad f^k-g^k = s^{k+1}d_A^k + d_B^{k-1}s^k).$$
\end{definition}

\begin{lemme}\label{grp.ht.qht}
Notons $QHt(A,B)$ (resp. $Ht(A,B)$) le sous-ensemble de
$\Hom_{\comp^*(\A)}(A,B)$, form\'e des morphismes quasi-homotopes
(resp. homotopes) au morphisme nul. Alors $QHt(A,B)$ (resp.
$Ht(A,B)$) est un sous-groupe ab\'elien de
$\Hom_{\comp^*(\A)}(A,B)$.
De plus $Ht(A,B)$ est un sous-groupe de $QHt(A,B)$.
\end{lemme}

\begin{definition}[comp. {\cite[déf. 1.3.6]{Bo}}]\label{cat.quasi-homot}
On d\'efinit la \SW{cat\'egorie des quasi-homotopies}, not\'ee $\qhtt^*(\A)$, par les r\`egles :
\begin{description}
\item[$\bullet$] $\Ob(\qhtt^*(\A))=\Ob(\comp^*(\A))$.
\item[$\bullet$] Pour toute paire d'objets $A$ et $B$ de
$\qhtt^*(\A)$,
$\dpl{\Hom_{\qhtt^*(\A)}(A,B)  =
\frac{\Hom_{\comp^*(\A)}(A,B)}{QHt(A,B)}}$.
\item[$\bullet$] Pour tout triplet d'objets $A$, $B$ et $C$ de $\qhtt^*(\A),
$\begin{eqnarray*} \Hom_{\qhtt^*(\A)}(A,B) \times
\Hom_{\qhtt^*(\A)}(B,C) &
\longrightarrow & \Hom_{\qhtt^*(\A)}(A,C) \\
(\overlined{f},\overlined{g}) &\longmapsto& \overlined{gf}
\end{eqnarray*}
o\`u $\overlined{f}$ d\'esigne  la classe de $f \in
\Hom_{\comp^*(\A)}(A,B)$ modulo $QHt(A,B)$.
\end{description}
\end{definition}

\begin{definition}\label{cat.homot}
On d\'efinit la \SW{cat\'egorie d'homotopies}, not\'ee $\htt^*(\A)$, par les r\`egles :
\begin{description}
\item[$\bullet$] $\Ob(\htt^*(\A))=\Ob(\comp^*(\A))$.
\item[$\bullet$] Pour toute paire d'objets $A$ et $B$ de $\htt^*(\A)$,
$\dpl{\Hom_{\htt^*(\A)}(A,B)  =
\frac{\Hom_{\comp^*(\A)}(A,B)}{Ht(A,B)}}$.
\item[$\bullet$] Pour tout triplet d'objets $A$, $B$ et $C$ de $\htt^*(\A)$,
\begin{eqnarray*} \Hom_{\htt^*(\A)}(A,B) \times
\Hom_{\htt^*(\A)}(B,C) &
\longrightarrow & \Hom_{\htt^*(\A)}(A,C) \\
(\overline{f},\overline{g}) &\longmapsto& \overline{gf}
\end{eqnarray*}
o\`u $\overline{f}$ d\'esigne  la classe de $f \in
\Hom_{\comp^*(\A)}(A,B)$ modulo $Ht(A,B)$.
\end{description}
\end{definition}

\begin{remarque}
Soit $\A$ une cat\'egorie additive. Alors la cat\'egorie
$\qhtt^*(\A)$ (resp. $\htt^*(\A)$) est
additive.
\end{remarque}

\subsection{Structures de poids.}

Soit $\CC$ une catégorie triangulée. Rappelons la définition de structure de poids.

\begin{definition}[{comp. \cite[déf. 1.1.1]{Bo}}]\label{def.WS}
On dira que $w=(\CCn{0},\CCp{0})$, où $\CCn{0}$, $\CCp{0}\subset \CC$, est une \SW{structure de poids} sur $\CC$, notée $w/\CC$, si les axiomes suivants sont satisfaits :
\begin{description}
\item[$(SP1)$. Stabilité par rétractes.] Les catégories $\CCn{0}$ et $\CCp{0}$ sont stables par rétractes.
\item[$(SP2)$. Semi-invariance avec respect des translations.] $$\CCn{0}[-1]\subset\CCn{0},\quad\CCp{0}[1]\subset\CCp{0}.$$
\item[$(SP3)$. Orthogonalité faible.] $$\CCn{0}\subset\CCp{0}[1]^\bot.$$
\item[$(SP4)$. Filtration par le poids.] Pour tout $X\in \CC$, il existe un $A\in \CCn{0}$, $B\in \CCp{1}$ et un triangle distingué $\tridis{A}{X}{B}$. On appelle un tel triangle une \SW{filtration par le poids} de $X$.
\end{description}
Pour tout $n\in \Z$, on note
$$\CCn{n}:=\CCn{0}[n],\quad \CCp{n}:=\CCp{0}[n],\quad\CCe{n}:=\CCn{n}\cap\CCp{n}.$$
On appelle $\CCe{0}$ le \SW{c{\oe}ur} de la structure de poids. 
\end{definition}

\`A noter que la notion de structure de poids fut indépendamment introduite par Pauksztello dans \cite{Pau} alors appelée \SW{co-$t$-structure}.
Dans la suite on munit $\CC$ d'une structure de poids $w/\CC$.

\begin{proposition}[comp. {\cite[prop. 1.3.3.3, 1.3.3.7]{Bo}}]\label{coeur.scinde}
Les catégories $\CCn{0}$, $\CCp{0}$ et $\CCe{0}$ sont stables par extensions. De plus, si $\tridis{A}{X}{B}$ est un triangle distingué de $\CC$ tel que $A\in\CCe{0}$ et $B\in \CCe{0}$ alors $X=A\oplus B$.
\end{proposition}

\begin{proof}
Puisque, par orthogonalité forte, les catégories $\CCn{0}$ et $\CCp{0}$ s'expriment comme des catégories orthogonales, elles sont stables par extensions (\cf \cite[lm. 1.7]{Heb}). Enfin, $\Hom_\CC(B,A[1])=0$ (orthogonalité faible) ce qui justifie que le triangle de l'énoncé se scinde.
\end{proof}

\begin{definition}[comp. {\cite[déf. 1.2.1]{Bo}}]
Soit $X\in \CC$. On dira que $X$ est \SW{born\'e inf\'erieurement} (resp. \SW{born\'e
sup\'erieurement}) s'il existe un entier $n\in \Z$ tel que $X\in\CCp{n}$ (resp. $\CCn{n}$). 
Lorsque $X$ est born\'e sup\'erieurement et inf\'erieurement, on dit qu'il est \SW{born\'e}.

Si tous les objets de $\CC$ sont born\'es (resp.
bornés sup\'erieurement, resp. bornés inférieurement), on dit que $w$ est \SW{born\'ee}
(resp.
\SW{bornée sup\'erieurement}, resp. \SW{bornée inférieurement}). 
\end{definition}

\begin{notation}
Soit $\tridis{A}{X}{B}$ un triangle distingué de $\CC$ tel que $A\in \CCn{n}$ et $B\in\CCp{n+1}$, on note alors 
$$\Xn{n}:=A,\qquad\text{et}\qquad\Xp{n+1}:=B.$$
D'aprés l'axiome $(SP4)$, il existe, pour tout $X\in \Z$ et $n\in \Z$, un triangle de la forme
$\tridis{\Xn{n}}{X}{\Xp{n+1}}$ ; cependant, on prendra garde que les objets $\Xn{n}\in \CCn{n}$ et $\Xp{n+1}\in \CCp{n+1}$ ne sont pas déterminés de manière unique.
\end{notation}

\begin{proposition}[Principe de $w$-distributivité] \label{w.distrib} Soient $a$, $b$, $c$ des entiers relatifs et $X\in \CC$. On peut choisir
$$\obn{X[a+c]}{b+c}=\obn{X[a]}{b}[c]\qquad \text{et}\qquad \obp{X[a+c]}{b+c}=\obp{X[a]}{b}[c].$$
\end{proposition}

\begin{proof}
On a le triangle $\tridis{\obn{X[a]}{b}[c]}{X[a+c]}{\obp{X[a]}{b+1}[c]}$ et $\obn{X[a]}{b}[c]\in \CCn{b+c}$, $\obp{X[a]}{b+1}[c]\in \CCp{b+c+1}$. Ainsi le triangle précédent correspond à $\tridis{\obn{X[a+c]}{b+c}}{X[a+c]}{\obp{X[a+c]}{b+c+1}}$.
\end{proof} 

\begin{definition}[{comp. \cite[déf. 1.2.1.VI]{Bo2}}]
Soient $\CC$ et $\CC'$ des catégories triangulées, $c/\CC$, $c'/\CC'$ des structures de poids et $F:\morph{\CC}{\CC'}$ un foncteur de catégories triangulées.
\begin{description}
\item[$\bullet$] On dira que $F$ est \SW{$w$-exact à gauche} si $F$ transforme les objets de $\CC_{c \leqslant0}$ en objets de $\CC'_{c' \leqslant0}$.
\item[$\bullet$] On dira que $F$ est \SW{$w$-exact à droite} si $F$ transforme les objets de $\CC_{c \geqslant0}$ en objets de $\CC'_{c' \geqslant0}$.
\item[$\bullet$] On dira que $F$ est \SW{$w$-exact} s'il est $w$-exact à gauche et à droite.
\end{description}
\end{definition}

\subsection{Complexe de poids.}

\begin{proposition}[comp. {\cite[prop. 1.5.6.2]{Bo}}]\label{suitededefdeCP}
Soient $n\in \Z$ et $X\in \CC$. Alors pour tout $k\in\Z$, il existe un objet $\W^n_\CC(X)^k\in \CCe{n}$ et des morphismes
$$\alpha^k_+ : \morph{\obn{X[k]}{n-1}}{\obn{X[k]}{n}}\qquad\text{et}\qquad \alpha^k_- : \morph{\obp{X[k]}{n}}{\obp{X[k]}{n+1}}$$
tels que l'on ait les triangles distingués suivants 
$$\xymatrix@R=0cm{\obn{X[k]}{n-1}\ar[r]^{\alpha^k_-}&\obn{X[k]}{n}\ar[r]^{\beta^k_-}&\W^n_\CC(X)^k\ar[r]^{\gamma^k_-}&\obn{X[k]}{n-1}[1],\\
\obp{X[k]}{n}\ar[r]_{\alpha^k_+}&\obp{X[k]}{n+1}\ar[r]_{\beta^k_+}&\W^n_\CC(X)^k[1]\ar[r]_{\gamma^k_+}&\obp{X[k]}{n}[1].}$$
\end{proposition}

\begin{proof}
Par orthogonalité, on peut appliquer \cite[prop. 1.1.9, prop. 1.1.11]{BBD} pour obtenir les flèches pointillées du diagramme
$$\xymatrix@C=2cm{\obn{X[k]}{n-1}\ar@{-->}[d]^{\alpha^k_-}\ar[r]&X[k]\ar[d]^{\Id_{X[k]}}\ar[r]&\obp{X[k]}{n}\ar@{-->}[d]^{\alpha^k_+}\ar[r]^{+1}&\\
\obn{X[k-1]}{n-1}[1]\ar[r]\ar@{-->}[d]&X[k]\ar[r]\ar[d]&\obp{X[k-1]}{n}[1]\ar@{-->}[d]\ar[r]^{+1}&\\
\W^n_\CC(X)^k\ar@{-->}[r]\ar@{-->}[d]^{+1}&0\ar@{-->}[r]\ar[d]^{+1}&\W^n_\CC(X)^k[1]\ar@{-->}[r]^{+1}\ar@{-->}[d]^{+1}&\\
&&&&}$$
où toutes les lignes et toutes les colonnes sont des triangles distingués. 

\noindent Le triangle distingué
$\tridis{\obn{X[k-1]}{n-1}[1]}{\W^n_\CC(X)^k}{\obn{X[k]}{n-1}[1]}$
justifie, par stabilité par extensions, que  $\W^n_\CC(X)^k\in \CCn{n}$.
De même, le triangle 
$\tridis{\obp{X[k-1]}{n}}{\W^n_\CC(X)^k}{\obp{X[k]}{n}}$
justifie que $\W^n_\CC(X)^k\in \CCp{n}$. 
Enfin on applique le principe de $w$-distributivité pour simplifier $\obn{X[k-1]}{n-1}[1]$ et $\obp{X[k-1]}{n}[1]$ (\cf \ref{w.distrib}).
\end{proof}

\begin{remarque}\label{w.distrib2}
Supposons que la structure de poids $w$ soit bornée et prenons un objet $X\in \CC$. Il existe donc des entiers $a$ et $b$ tels que $X\in \CCp{a}\cap\CCn{b}$ (on peut supposer $a\geqslant b$ par orthogonalité, à moins que l'objet $X$ soit $0$). On observe alors 
\begin{multicols}{2}
\begin{center}
si $b\leqslant k$ alors $\Xn{k}=X$,\\
si $b< k$ alors $\Xp{k}=0$,\\
si $a\geqslant k$ alors $\Xp{k}=X$,\\
si $a> k$ alors $\Xn{k}=0$.
\end{center}
\end{multicols}
Ainsi, d'aprés la proposition précédente on trouve deux filtrations de l'objet $X$ :
$$\xymatrix@C=0.5cm@R=0cm{\cdots\ar[r]&0\ar[r]&\Xn{a}\ar[r]&\Xn{a+1}\ar[r]&\cdots\ar[r]&\Xn{b-2}\ar[r]&\Xn{b-1}\ar[r]&X\ar[r]&X\ar[r]&\cdots\\
\cdots\ar[r]&X\ar[r]&X\ar[r]&\Xp{a+1}\ar[r]&\Xp{a+2}\ar[r]&\cdots\ar[r]&\Xp{b-1}\ar[r]&\Xp{b}\ar[r]&0\ar[r]&\cdots}$$
(tour de Postnikov, \cf \cite[déf 1.5.8]{Bo}).
On peut voir $\Xn{k}$ et $\Xp{k}$ comme des approximations de $X$ de sorte que $\W_\CC^n(X)^0$ et $\W_\CC^n(X)^0[1]$ mesure l'erreur commise entre deux approximations.

Le principe de $w$-distributivité permet les identifications suivantes 
$$\W_\CC^n(X[i])^k=\W_\CC^n(X)^{k+i},\qquad\W_\CC^n(X)^k[i]=\W_\CC^{n+i}(X)^{k+i}$$
quelques soient les entiers $k$, $i$ et $n$. De sorte qu'il suffit d'étudier $\W_\CC^0(X)^k$ que l'on note $\W_\CC(X)^k$.
\end{remarque}

\begin{proposition}[comp. {\cite[déf. 2.2.1, prop. 2.2.2.1]{Bo}}]\label{lm.diff.mult}
Soient $k\in\Z$ et $X\in \CC$. Notons $\varphi(X)$ le morphisme $\morph{\Xp{0}}{\obn{X[1]}{0}}$. On a les identifications suivantes
$$d_{\W_\CC(X)}^k :=\beta^{k+1}_-\varphi(X[k])(\gamma_+^{k}[-1])=(\beta^{k+1}_+[-1])(\gamma^{k}_+[-1])=-\beta^{k+1}_-\gamma^{k}_-.$$
De plus $\left(\W_\CC(X)^k,d_{\W_\CC(X)}^k\right)_{k\in \Z}$ définit un complexe de $\CCe{0}$. On l'appelle le \SW{complexe de poids} de $X$.
\end{proposition}

\begin{proof}
Les différentes identifications de $d_{\W_\CC(X)}^k$ viennent du diagramme de définition de $\W_\CC(X)^k$. On rappelle qu'il est obtenu par le lemme des $9$ (\cite[prop. 1.1.11]{BBD}) d'où, en isolant la partie en bas à droite (à un décalage près), on a le diagramme
$$\xymatrix@C=2cm{\obp{X[k-1]}{0}\ar[r]^{\beta^{k}_+[-1]}\ar[d]_{\varphi(X[k-1])}&\W_\CC(X)^k\ar[r]^{\gamma^{k}_+[-1]}\ar[d]_{\Id_{\W_\CC(X)^k}}&\obp{X[k]}{0}\ar[d]^{\varphi(X[k])}\\
\obn{X[k]}{0}\ar[r]_{\beta^{k}_-}&\W_\CC(X)^k\ar[r]_{\gamma^{k}_-}&\obn{X[k+1]}{0}[1]}$$
qui n'est pas commutatif. Le carré de gauche est commutatif d'où $\beta^{k}_+[-1]=\beta^{k}_-\varphi(X[k-1])$ tandis que le carré de droite est anti-commutatif d'où $\gamma^{k}_-=-\varphi(X[k])(\gamma^{k}_+[-1])$. La dernière assertion résulte du fait que la composée de deux morphismes consécutifs d'un triangle distingué est nulle.
\end{proof}

\begin{theoreme}[comp. {\cite[thm. 3.2.2.II]{Bo}}]\label{def.fonct.complexe}
La transformation donnée par la règle
$$\xymatrix@C=1.5cm@R=0cm{
\W_\CC : \CC\ar[r]&\qhtt(\CCe{0})\\
X\ar@{|->}[r]&\W_\CC(X)
}$$
définit un foncteur additif.
\end{theoreme}

\begin{proposition}[comp. {\cite[thm. 3.3.1.IV]{Bo}}]\label{prop.Bo.3.3.1.IV}
Soient $a\in \Z$ et $X\in \CCn{a}$ (resp. $X\in \CCp{a}$) alors $$\W_\CC(X)\in \qhtt^{\geqslant -a}(\CCe{0}) \quad (\text{resp. } \W_\CC(X)\in \qhtt^{\leqslant -a}(\CCe{0})).$$
\end{proposition}

\begin{proof}
Il s'agit de voir que l'on peut choisir $\W_\CC(X)^k=0$ lorsque $k<-a$ ce qui peut se lire grace à \ref{w.distrib2}. 
\end{proof}

On renvoie à \cite[app.]{Bei.Fcat} pour la définition et les propriétés des $f$-catégories.

\begin{theoreme}[Beilinson-Bondarko ; comp. {\cite[\S 8.4]{Bo}}]\label{thm.BeiBon}
Soit $\CC$ une $f$-catégorie munie d'une structure de poids bornée $w/\CC$. Il existe un foncteur de catégories trinagulées qui factorise $\W_\CC : \morph{\CC}{\qhtt^b(\CCe{0})}$. On le note également $\W_\CC$.
$$\xymatrix@C=1.5cm@R=0.2cm{&\qhtt^b(\CCe{0})\\
\CC\ar[ru]^-{\W_\CC}\ar[rd]_-{\W_\CC}&\\
&\htt^b(\CCe{0})\ar@{->>}[uu]}$$
\end{theoreme}

\section{Motifs et poids.}
\addcontentsline{toc}{section}{\protect\numberline{\thesection} Motifs et poids.}

\subsection{Motifs de Beilinson.}\label{sec.rap.mot}

Dans la suite on choisit de se placer dans la catégorie des motifs de Beilinson $\DMB{S}$ suivant Cisinski-Déglise (\cite{CD}). Nous présentons un échantillon de propriétés de cette catégorie. En plus de \cite{CD}, on peut également se référer à \cite[\S 2]{Heb} et \cite[prop. 1.1.2]{Bo2}.

\begin{description}
\item[1. ] On a les opérations de Grothendieck (\cite[part 1]{CD}): si $f$ désigne un morphisme de schémas
$$(f_\sharp, \quad\text{si $f$ est lisse})\;\rightleftarrows\; f^*\;\rightleftarrows\;f_*\qquad f_! \;\rightleftarrows\; f^!\qquad \otimes_S\;\rightleftarrows\; \HHom_S$$
\item[2. ]
\setlength{\columnseprule}{1pt}
\begin{multicols}{2}
Soit
$$\xymatrix{W\immouv[r]^{j'}\ar[d]_{f}\ar@{}[rd]|\square&Y\ar[d]^f\\
U\immouv[r]_j&X}$$
un diagramme Nisnevich-distingué (\ie $j$ est une immersion ouverte et $f$ est un morphisme étale tel que $f^{-1}(X-U)_{red}\simeq(X-U)_{red}$)

\noindent(resp.
$$\xymatrix{T\immfer[r]^{i'}\ar[d]_{p'}\ar@{}[rd]|\square&Y\ar[d]^p\\
Z\immfer[r]_i&X}$$
un diagramme propre $cdh$-distingué (\ie $i$ est une immersion fermée et $p$ est un morphisme propre et surjectif tel que $p^{-1}(X-Z)_{red}\simeq(X-Z)_{red}$)).
\end{multicols}
D'aprés \cite[prop. 5.2.13]{CD} (resp. le théorème de Voevodsky \cite[thm. 3.3.7, \S 3.3.8]{CD}) on a le triangle
\begin{eqnarray*}
&&\tridis{(fj')_!(fj')^*M}{f_!f^*M\oplus j_!j^*M}{M}\\
(\text{resp.}&&\tridis{M}{ p_!p^*M\oplus i_!i^*M}{(pi')_!(pi')^*M}\quad).
\end{eqnarray*}
pour tout motif $M\in \DMB{X}$. On peut dualiser ces constructions (\cite[\S 3.3.3]{CD}):
\begin{eqnarray*}
&&\tridis{M}{f_*f^!M\oplus j_*j^!M}{(fj')_*(fj')^!M}\\
(\text{resp.}&&\tridis{(pi')_*(pi')^!N}{ p_*p^!M\oplus i_*i^!M}{M}\quad).
\end{eqnarray*}
\item[3. ] On a la pureté relative et absolue : si $f:\morph{X}{Y}$ est un morphisme de schémas lisse de dimension relative $d$ (resp. une immersion fermée de codimension $-d$ entre schémas réguliers) alors (\cite[thm. 2.4.15.$(iii)$, rm. 2.4.16, thm. 13.4.1]{CD})
$$f^!\Un_Y=\Un_X(d)[2d].$$
\item[4. ] Si $U$ est un ouvert de $S$ de fermé complémentaire $Z$, alors en notant $j:\immop{U}{S}$ et $i:\immcl{Z}{S}$ les immersions canoniques, on a les triangles distingués de localisation (\cite[prop. 2.3.3, thm. 2.2.14.$(2)$]{CD})
$$\tridis{j_!j^*M}{M}{i_!i^*M},\qquad\qquad\tridis{i_*i^!M}{M}{j_*j^!M}$$
pour tout motif $M\in \DMB{S}$.
\item[5. ] On a la $h$-descente : considérons le diagramme suivant, où les carrés sont cartésiens
$$\xymatrix{
Z'\immfer[r]\ar@/_/[rd]^a\ar[d]&X'\ar[d]^p&\immouv[l]\ar[d]U'\\
Z\immfer[r]_i&X&\immouv[l]U}$$
et où $p$ est une altération de Galois de groupe $G$ telle que génériquement $\morph{X'/G}{X}$ est fini, surjectif et radiciel (en particulier $X'$ et $X$ ont la même dimension), $U$ est normal et $\morph{U'}{U}$ est fini alors on a les triangles distingués (\cite[thm. 14.3.7, rm. 14.3.8]{CD})\\
$$\tridis{M}{i_!i^*M\oplus \left(p_!p^*M\right)^G}{\left(a_!a^*M\right)^G},\qquad\qquad\tridis{\left(a_*a^!M\right)_G}{i_*i^!M\oplus \left(p_*p^!M\right)_G}{M},$$
pour tout motif $M\in \DMB{S}$ (où $R_G$ et $R^G$ désignent des rétractes complémentaires l'un de l'autre ; $R=R_G\oplus R^G$).
\item[6. ] Lorsque $f:\morph{X}{S}$ est lisse, on pose $M_S(X):=f_\sharp\Un_X$. La catégorie des motifs \SW{constructibles} (\cite[déf. 1.4.7]{CD}) est la plus petite catégorie épaisse contenant
$$\DMB{S}\supset\G_S:=\left\{M_S(X)(n)\big|n\in \Z, f:\morph{X}{S}\hbox{ lisse}\right\}.$$
De plus les opérations de Grothendieck respectent les objets constructibles (\cite[thm. 14.1.31]{CD}).
\item[7. ] On note $\HB^n(S,\Q(m))$ la \SW{cohomologie de Beilinson} définie par (\cite[déf. 13.2.13]{CD}),
$$\HB^n(S,\Q(m)):=\Hom_{\DMBc{S}}(\Un_S,\Un_S(m)[n]),$$
quelques soient les entiers $m$ et $n$.
\item[8.] Notons $\pi:\morph{S}{B}$ le morphisme structural de $S$ (\cf introduction). Pour tout $f:\morph{X}{S}\in \Sch/S$, le foncteur de dualité locale est défini par (\cite[\S 14.3.30]{CD})
$$\xymatrix@R=0cm@C=1.7cm{\D_X:=\DMBc{X}^\opp\ar[r]&\DMBc{X}\\
M\ar@{|->}[r]&\HHom_X(M,(\pi f)^!\Un_B).}$$ 
Si $S$ est un schéma régulier, $\D_X$ échange $*$ et $!$ (\cite[cor. 14.3.31.$(d)$]{CD}) et pour tout motif constructible $M$ on a un isomorphisme $\D_X^2M\simeq M$ (\cite[cor. 14.3.31.$(b)$]{CD}).
\item[9. ] Soit $$\DMB{S}\supset\H_S:=\left\{f_!\Un_X(x)[2x]\big|x\in \Z,\;f:\morph{X}{S} \textrm{ propre à domaine régulier}\right\}.$$
\begin{description}
\item[$(i)$.] Il existe une unique structure de poids $W/\DMB{S}$ telle que les petites sommes d'objets de $\H_S$ sont dans $\DMBe{S}$.
\item[$(ii)$.] Il existe une unique structure de poids bornée $w/\DMBc{S}$ telle que $\H_S\subset\DMBce{S}$. Précisément $\DMBce{S}$ est l'enveloppe des rétractes de la plus petite catégorie additive contenant $\H_S$.
\item[$(iii)$.] Le foncteur d'inclusion canonique ${\DMBc{S}}\hookrightarrow{\DMB{S}}$ est $w$-exact.
\end{description}
De plus les opérations de Grothendieck qui sont des adjoints à gauche (resp. à droite,) sont $w$-exactes à gauche (resp. à droite) (\cite[thm. 3.3, thm. 3.7]{Heb})
\begin{quot}
Justifions que la structure de poids sur $\DMBc{S}$ est bornée. De manière générale, les structures de poids issues du théorème de construction de Bondarko (\cf \cite[thm. 4.3.2.II.1]{Bo}) sont bornées. Soit $\CC$ une catégorie triangulée satisfaisant la condition $(ii)$ de \cite[thm. 1.9]{Heb}. Montrons par exemple que les objets de $\CC$ sont bornés inférieurement : par hypothèse $\H$ engendre $\CC$ il suffit donc de montrer par récurrence sur $n\in \N$, que le objets de $\dpl{\H_n:=\SE_\CC^n\left(\bigcup_{k\in \Z}\H[k]\right)}$ sont bornés inférieurement. On rappelle que la construction nous donne 
$\dpl{\CCp{0}=\Kar\left(\SE_\CC\left(\bigcup_{k\geqslant 0}\H[k]\right)\right)}$.
Soit $X\in\H_0$ alors $X\in \H[i]$, c'est à dire qu'il s'\'ecrit $A[i]$ pour un objet $A$
de $\H$. Si $i$ est positif alors $A[i]\in\CCp{0}$. Si $i$ est n\'egatif alors
$A\in\H[-i]\subset\CCp{0}$ et donc
$X=A[i]\in\CCp{i}$.

Si $X\in\H_{n+1}$ alors il existe un triangle
distingu\'e $\tridis{A}{X}{B}$ o\`u $A$ et $B$ sont des objets de
$\H_n$ qui, par hypoth\`ese de r\'ecurrence, sont born\'es
inf\'erieurement. C'est \`a dire qu'il existe deux entiers $a$ et
$b$ tels que $A\in\CCp{a}$ et $B\in\CCp{b}$. On peut supposer que $a\leqslant b$ ; en particulier
$B\in \CCp{a}$ (axiome $(SP2)$). La conclusion suit de la stabilité par extension de $\CCp{a}$.
\end{quot}
\item[10. ] Il existe un foncteur 
$$\W_{\DMB{S}} : \morph{\DMB{S}}{\qhtt(\DMBe{S})}$$
tel que pour tout $M\in \DMBe{S}$, $\W_S(M)=M$ (concentré en degré $0$ ; \cf \cite[prop. 3.3.1]{Bo2},  \ref{def.fonct.complexe}). On note $\Chow{S}$ le c{\oe}ur de $w/\DMBc{S}$ et $\W_S$ la restriction du foncteur $\W_{\DMB{S}}$ à $\DMBc{S}$.
$$\W_S : \morph{\DMBc{S}}{\htt^b(\Chow{S})}$$
\begin{quot}
Justifions brièvement : par définition (\cf \cite[déf. 13.2.1]{CD}) la catégorie $\DMBc{S}$ s'identifie à une  sous-catégorie pleine (localisation) d'une catégorie dérivée d'une catégorie abélienne qui est une $f$-catégorie (\cf \cite[app. A 2]{Bei.Fcat}). Ainsi $\DMBc{S}$ est une $f$-catégorie (\cf \cite[prop. 2.2]{Wild.Fcat}) munie d'une structure de poids bornée. On applique \ref{thm.BeiBon}.
\end{quot}
\end{description}

On se référera à ces rappels avec la notation $(rap. i)$.

\subsection{Dualité.}
On rappel que l'opérateur de dualité est un foncteur de catégories triangulées contravariant.

\begin{lemme}\label{lm.Brad}
Soient $f:\morph{X}{S}$ un morphisme de schémas, $M\in \DMBc{X}$ et $(a,b)\in \Z^2$, 
$$\D_X(M(a)[b])=\D_XM(-a)[-b].$$
\end{lemme}

\begin{proof}
Pour tout motif $N\in \DMBc{X}$, 
\begin{eqnarray*}
\Hom_{\DMBc{X}}(N,\D_X(M(a)[b]))&=&\Hom_{\DMBc{X}}(N,\HHom_X(M(a)[b],(\pi f)^!\Un_B))\\
&=&\Hom_{\DMBc{X}}(N\otimes_X M(a)[b],(\pi f)^!\Un_B)\\
&=&\Hom_{\DMBc{X}}(N(a)[b]\otimes_X M,(\pi f)^!\Un_B)\\
&=&\Hom_{\DMBc{X}}(N(a)[b],\HHom_X(M,(\pi f)^!\Un_B))\\
&=&\Hom_{\DMBc{X}}(N(a)[b],\D_XM)\\
&=&\Hom_{\DMBc{X}}(N,\D_XM(-a)[-b]),
\end{eqnarray*}
ce qui permet de conclure.
\end{proof}

\begin{theoreme}\label{thm.conj.Heb}
Soit $f:\morph{X}{S}$ un morphisme projectif entre schémas réguliers ; alors $f_*f^!\Un_S\in\Chow{S}$.
\end{theoreme}

\begin{proof}
Puisque $S$ est régulier, $\Un_S\in \Chow{S}$ (\cite[thm. 3.7.$(v)$]{Heb}) et la $w$-exactitude à droite de $f_*$ et $f^!$ implique $f_*f^!\Un_S\in \DMBcp{S}$. Pour montrer que $f_*f^!\Un_S\in \DMBcn{S}$ on utilise \cite[rm. 3.5]{Heb} : on montre que $\Hom_{\DMBc{S}}(f_*f^!\Un_S,g_!\Un_Y(a)[b])=0$ pour tout morphisme $g:\morph{Y}{S}$ propre à domaine régulier et tout $(a,b)\in \Z^2$ tel que $b>2a$

Puisque $f$ est projectif, on a une factorisation en une immersion fermée et un morphisme lisse, 
où $P$ est régulier. On obtient alors :
\begin{multicols}{2}
\hspace*{3.25cm}$\Hom_{\DMBc{S}}(f_*f^!\Un_S,g_!\Un_Y(a)[b])$
$$\xymatrix@R=0.3cm@C=1cm{X\ar[rr]^f\immfer[dr]_i&&S\\
&P\ar[ru]_s&}\qquad\qquad
\xymatrix@R=0.3cm{Z\ar@{}[rd]|{\square}\ar[d]_{g'}\ar[r]^{s'}&Y\ar[d]^g\\
P\ar[r]_s&S}$$
\begin{eqnarray*}
&=&\Hom_{\DMBc{S}}(f_!f^!\Un_S,g_!\Un_Y(a)[b])\\
&=&\Hom_{\DMBc{S}}(s_!i_!i^!s^!\Un_S,g_!\Un_Y(a)[b])\\
&=&\Hom_{\DMBc{P}}(i_!i^!s^!\Un_S,s^!g_!\Un_Y(a)[b])\\
&=&\Hom_{\DMBc{P}}(i_!i^!s^!\Un_S,g'_!s'^!\Un_Y(a)[b]).
\end{eqnarray*}
\end{multicols}
On utilse la pureté relative $(rap. 3)$, pour obtenir $s'^!\Un_Y=\Un_Z(d')[2d']$ où $d'$ est la dimension relative de $s'$ et $s^!\Un_S=\Un_P(d)[2d]$ où $d$ est la dimension relative de $s$, puis on utilise la pureté absolue pour obtenir $i^!\Un_P=\Un_X(-c)[-2c]$ où $c$ est la codimension de $X$ dans $P$. Finalement, 
$$\Hom_{\DMBc{S}}(f_*f^!\Un_S,g_!\Un_Y(a)[b])=\Hom_{\DMBc{P}}(i_!\Un_X(d-c)[2(d-c)],g'_!\Un_Z(a+d')[b+2d'])$$
Mais $i_!\Un_X(d-c)[2(d-c)]\in \Chow{P}$ et $g'_!\Un_Z(a+d')[b+2d']\in\Chow{P}[b-2a]\subset\DMBc{P}_{w\geqslant 1}$. On conclut par orthogonalité.
\end{proof}

\begin{lemme}\label{rm.dual}
Supposons que $S$ est un schéma lisse au dessus de $B$, alors $\D_S(\Un_S)=\Un_S(s)[2s]$ pour $s$ la dimension relative de $\pi: \morph{S}{B}$.
\end{lemme}

\begin{proof}
Cela suit de la pureté relative $(rap. 3)$, $$\D_S(\Un_S)=\HHom_S(\Un_S,\pi^!\Un_B)=\HHom_S(\Un_S,\Un_S(s)[2s])=\HHom_S(\Un_S,\Un_S)(s)[2s]=\Un_S(s)[2s].$$
\end{proof}

Dans la suite de cette section, on suppose que $S$ est régulier et lisse au dessus de $B$.

\begin{corollaire}
Soit $f:\morph{X}{S}$ un morphisme propre à domaine régulier ; alors $f_*f^!\Un_S\in\Chow{S}$.
\end{corollaire}

\begin{proof}
D'aprés le lemme de Chow motivique (\cite[lm. 3.1]{Heb}) $f_!f^*\Un_S$ est un rétracte d'un motif $p_!p^*\Un_S$ pour un certain morphisme $p:\morph{X_0}{S}$ projectif à domaine régulier. L'opérateur de dualité étant additif, on en déduit que $f_*f^!\Un_S=\D_S(f_!f^*\Un_S)(-s)[-2s]$ est un rétracte de $p_*p^!\Un_S=\D_S(p_!p^*\Un_S)(-s)[-2s]$. On conclut avec le théorème précédent et la stabilité par rétractes du c{\oe}ur.
\end{proof}

\begin{corollaire}\label{prop.dual.conj}
Quelque soit $n\in \Z$, le foncteur $\D_S$ induit une équivalence de catégories entre $\DMBc{S}_{w\leqslant n}^\opp$ et $\DMBc{S}_{w\geqslant -n}$.
\end{corollaire}

\begin{proof}
D'aprés \cite[cor. 3.8]{Heb}, on sait que le foncteur $\D_S$ induit un foncteur $\morph{\DMBc{S}^{\opp}_{w\leqslant n}}{\DMBc{S}_{w\geqslant -n}}$. Puisque $\D_S^2M\simeq M$, il suffit de montrer que $\D_S$ induit un foncteur $\morph{\DMBc{S}^{\opp}_{w\geqslant n}}{\DMBc{S}_{w\leqslant -n}}$ ; c'est-à-dire qu'il faut montrer que si $M\in \DMBc{S}^{\opp}_{w\geqslant n}$ alors $\D_SM\in \DMBc{S}^{\opp}_{w\leqslant -n}$. Via \ref{lm.Brad}, il suffit de prouver le résultat pour $n=0$. 
\begin{description}
\item[$\bullet$] Soit $M\in \Pos_S$, c'est-à-dire que $M=f_!f^*\Un_S(a)[b]$ pour un certain morphisme $f:\morph{X}{S}$ propre à domaine régulier et des entiers $a$ et $b$ tels que $b\geqslant 2a$. Alors, $\D_SM=\D_S(f_!f^*\Un_S(a)[b])=f_*f^!\Un_S(s-a)[2s-b]$ (\cf \ref{rm.dual}). Mais par le corollaire précédent $f_*f^!\Un_S(s-a)[2s-b]\in\Chow{S}[2a-b]\subset\DMBcn{S}$.
\item[$\bullet$] On va montrer par récurrence sur $n\in \N$ que les objets de $\H_n:=\SE_{\DMBc{S}}^n\left(\Kar(\Pos_S)\right)$ sont envoyés par $\D_S$ dans $\DMBcn{S}$. Pour $n=0$, il s'agit de le montrer pour les rétractes de $\Pos_S$, mais cela suit du point précédent et de la stabilité par rétractes.
Soit $M\in \H_{n+1}$, alors il existe un triangle distingué $\tridis{A}{M}{B}$ où $A$ et $B$ sont des objets de $\H_n$ ; appliquant $\D_S$ on obtient 
$\tridis{\D_SB}{\D_SM}{\D_SA}$. On conclut par stabilité par extensions (\cf \ref{coeur.scinde}).
\item[$\bullet$] D'aprés \cite[rm. 3.5]{Heb}, $\dpl{\DMBcn{S}=\Kar\left(\bigcup_{n\in\N}\H_n\right)}$. L'additivité de $\D_S$ permet de conclure.
\end{description}
\end{proof}

\subsection{Poids de motifs.}

\begin{theoreme}\label{thm.mesure.motifs}
Soit $f:\morph{X}{S}$ un morphisme de schémas tel que $\Dim(X)=d$, alors
$$f_!\Un_X\in \DMBc{S}_{w\geqslant-d}\cap\DMBcn{S}.$$
\end{theoreme}

\begin{proof}
Puisque $\Un_X$ est un objet de poids au plus $0$ (\cf \cite[thm. 3.7.$(v)$]{Heb}) et que $f_!$ est $w$-exact à gauche alors $f_!\Un_X\in \DMBcn{S}$. 
\begin{description}
\item[$\bullet$] Supposons dans un premiers temps que $f$ est propre et montrons par récurrence sur $d$ que $f_!\Un_X\in \DMBc{S}_{w\geqslant-d}$. 
On considère une altération de Galois, comme dans $(rap. 5)$ (qui existe en vertue de \cite[thm. 14.3.6]{CD}) avec $M=\Un_X$ pour obtenir le triangle distingué
$\tridis{\Un_X}{i_!\Un_Z\oplus \left(p_!\Un_{X'}\right)^G}{\left(a_!\Un_{Z'}\right)^G}$
qui donne en décalant et en composant par $f_!$ le triangle
$$\tridis{f_!\left(a_!\Un_{Z'}\right)^G[-1]}{f_!\Un_X}{(fi)_!\Un_Z\oplus f_!\left(p_!\Un_{X'}\right)^G}$$
Par construction $Z'$ et $Z$ sont de dimension au plus $d-1$ ; puisque $fa$ est propre on peut appliquer l'hypothèse de récurrence pour en déduire que 
$(fa)_!\Un_{Z'}\in\DMBc{S}_{w\geqslant-\Dim(Z')}\subset\DMBc{S}_{w\geqslant-d+1}$ (axiome $(SP2)$). Ainsi la stabilité par rétractes (axiome $(SP1)$) donne $f_!\left(a_!\Un_{Z'}\right)^G\in \DMBc{S}_{w\geqslant-d+1}$ d'où $f_!\left(a_!\Un_{Z'}\right)^G[-1]\in \DMBc{S}_{w\geqslant-d}$. De même $(fi)_!\Un_Z \in \DMBc{S}_{w\geqslant-d}$. Puisque $fp$ est propre et $X'$ régulier, $(fp)_!\Un_{X'}\in \DMBce{S}\subset\DMBc{S}_{w\geqslant-d}$ ; il en va de même pour le rétracte $f_!\left(p_!\Un_{X'}\right)^G$ de sorte que $(fi)_!\Un_Z\oplus f_!\left(p_!\Un_{X'}\right)^G\subset \DMBc{S}_{w\geqslant-d}$ via la stabilité par extensions ce qui permet également de conclure.
\item[$\bullet$] De manière générale, on raisonne également par récurrence sur $d$. On choisit une compactification $\overline{X}$ de $X$ telle que $\Dim(\overline{X})=\Dim(X)=d$ (\cf \cite[\S 4 thm. 2]{Nagata}). On a alors le diagramme commutatif
\begin{multicols}{2}
$$\xymatrix{X\ar[rd]_f\immouv[r]^j&\overline{X}\ar[d]_p&\partial\overline{X}\immfer[l]_i\ar[dl]^g\\
&S&}$$

\columnbreak

Composant le triangle de localisation $(rap. 4)$ par $p_!$ on arrive au triangle distingué
$\tridis{g_!\Un_{\partial\overline{X}}[-1]}{f_!\Un_X}{p_!\Un_{\overline{X}}}$.
Par hypothèse de récurrence $g_!\Un_{\partial\overline{X}}[-1]\in \DMBc{S}_{w\geqslant-d}$ et puisque $p$ est propre le point précédent donne $p_!\Un_{\overline{X}}\in \DMBc{S}_{w\geqslant-d}$. On conclut via la stabilité par extension.
\end{multicols}
\end{description}
\end{proof}

\begin{corollaire}\label{thm.mesure.motifs.2}
Supposons que $S$ est un schéma régulier et lisse au dessus de $B$.
Pour tout morphisme $f:\morph{X}{S}$, où $X$ est un schéma de dimension $d$, 
$$f_*f^!\Un_S\in \DMBcp{S}\cap\DMBc{S}_{w\leqslant d}.$$
\end{corollaire}

\begin{proof}
On applique $\D_S$ au motif $f_!f^*\Un_S$, puis \ref{thm.mesure.motifs} et \ref{prop.dual.conj}.
\end{proof}

\begin{remarque}
Pour $S=\Spec(k)$, pour un corps parfait $k$ admettant la résolution des singularités, le théorème \ref{thm.mesure.motifs} et le corollaire \ref{thm.mesure.motifs.2} étaient déjà connus (\cite[thm. 6.2.1]{Bo3}, \cite[cor. 1.14]{Wild}).
\end{remarque}

On peut avoir une partie du corollaire \ref{thm.mesure.motifs.2} avec moins d'hypothèse sur $S$.

\begin{definition}
On dira qu'un schéma $S$ est \SW{pseudo-régulier} si $\Un_S\in\Chow{S}$.
\end{definition}

\begin{remarque}
Si $S$ est un schéma régulier alors il est pseudo-régulier (\cf \cite[thm. 3.7.$(v)$]{Heb}). Il semblerait que la réciproque soit fausse ($S=\Spec(\C[X]/X^2)$).
\end{remarque}

Comme pour les schémas réguliers, on a une caractérisation relative des schémas pseudo-réguliers.

\begin{proposition}\label{prop.pseudo-reg.relatif}
Soit $f:\morph{S}{T}$ un morphisme lisse de schémas. Si $T$ est pseudo-régulier alors $S$ aussi. 
\end{proposition}

\begin{proof}
Cela suit de la $w$-exactitude de $f^*$ (\cite[thm. 3.7.$(ii')_c$]{Heb}) et de $f^*\Un_T=\Un_S$.
\end{proof}

\begin{proposition}\label{thm.mesure.motifs.3}
Soient $S$ un schéma pseudo-régulier et $f:\morph{X}{S}$ un morphisme de schémas. Alors
$$f_*f^!\Un_S\in \DMBcp{S}.$$
\end{proposition}

\begin{proof}
Cela suit de la $w$-exactitude à droite de $f_*$ et $f^!$.
\end{proof}

\begin{corollaire}\label{cor.ann.coh}
Soient $f:\morph{X}{S}$, $g:\morph{Y}{S}$ des morphismes de schémas, $d_X=\Dim(X)$, $d_Y=\Dim(Y)$, $d_S=\Dim(S)$ et $(a,b)\in \Z^2$.
\begin{description}
\item[$(i).$] Si $b>2a+d_Y$ alors
$\Hom_{\DMBc{S}}(f_!f^*\Un_S,g_!g^*\Un_S(a)[b])=0$.
\item[$(i)'$.] Si $b>2a+d_S$ alors $\HB^b(S,\Q(a))=0$.
\item[$(ii)$.] Si $S$ est pseudo-régulier et $b>2a$ alors
$\Hom_{\DMBc{S}}(f_!f^*\Un_S,g_*g^!\Un_S(a)[b])=0$.
\item[$(ii)'$.] Si $S$ est pseudo-régulier et $b>2a$ alors
$\HB^b(S,\Q(a))=0$.
\end{description}
Supposons de plus que $S$ est régulier et lisse au dessus de $B$.
\begin{description}
\item[$(iii)$.] Si $b>2a+d_X$ et $b<2a-d_Y$, $$\Hom_{\DMBc{S}}(f_*f^!\Un_S,g_*g^!\Un_S(a)[b])=\Hom_{\DMBc{S}}(f_!f^*\Un_S,g_!g^*\Un_S(-a)[-b])=0.$$
\item[$(iii)'$.] Si $b>2a$ et $b<2a-d_S$, $\HB^b(S,\Q(a))=0$.
\item[$(iv)$.] Si $b>2a+d_X+d_Y$ et $b<2a$, $$\Hom_{\DMBc{S}}(f_*f^!\Un_S,g_!g^*\Un_S(a)[b])=\Hom_{\DMBc{S}}(f_!f^*\Un_S,g_*g^!\Un_S(-a)[-b])=0.$$
\end{description}
\end{corollaire}

\begin{proof}
On utilise l'axiome d'orthogonalité des structures de poids pour établir les annulations $(i)$ et $(ii)$. Les annulations $(i)'$ et $(ii)'$ suivent respectivement de $(i)$ et $(ii)$ pour $f=g=\Id_S$. Le foncteur de dualité étant une auto-équivalence de catégorie auto-duale $(rap. 8)$,
\begin{eqnarray*}
\Hom_{\DMBc{S}}(f_*f^!\Un_S,g_*g^!\Un_S(a)[b])&=&\Hom_{\DMBc{S}}(\D_S^2(f_*f^!\Un_S),g_*g^!\Un_S(a)[b])\\
&=&\Hom_{\DMBc{S}}(\D_S(f_*f^!\Un_S),\D_S(g_*g^!\Un_S(a)[b]))\\
&=&\Hom_{\DMBc{S}}(f_!f^*\Un_S(s)[2s],g_!g^*\Un_S(s-a)[2s-b])\\
&=&\Hom_{\DMBc{S}}(f_!f^*\Un_S,g_!g^*\Un_S(-a)[-b]),
\end{eqnarray*}
ce qui implique $(iii)$ par orthogonalité. Le même raisonnement amène $(iv)$. Finalement $(iii)'$ suit de $(iii)$ avec $f=g=\Id_S$ mais aussi du lien entre cohomologie de Beilinson et $\KK$-théorie \cite[cor. 13.2.14]{CD} (le groupe de Quillen rationnel $\KK_n(S)_\Q$ est nul si $n<0$).
\end{proof}

\begin{remarque}
Les points $(i)'$ et $(ii)'$ du précédent corollaire généralise \cite[cor. 13.2.14]{CD}. 
\end{remarque}

Comme application du précédent corollaire, on va définir une structure de poids sur la catégorie des motifs de Tate.

\begin{corollaire}\label{cor.struct.Poids.Tate}
Soit $S$ un schéma pseudo-régulier. Notons $\DMBcT{S}$ la catégorie des \SW{motifs de Tate} (\cf \cite[\S 3.3]{Esn.Lev}) : c'est la sous-catégorie pleine et triangulée de $\DMBc{S}$ engendrée par
$$\DMBc{S}\supset \T_S:=\left\{\Un_S(x)[2x]\big| x\in \Z\right\}.$$
Il existe une unique structure de poids bornée $\tate/\DMBcT{S}$ telle que $\T_S\subset\DMBcTe{S}$.

De plus, l'inclusion canonique $\DMBcT{S}\hookrightarrow\DMBc{S}$ est $w$-exacte.
\end{corollaire}

\begin{proof}
On applique le théorème de construction de Bondarko (\cite[cor. 1.20]{Heb}) ; l'hypothèse $(ii)$ de \textit{loc. cit.} suivant de \ref{cor.ann.coh}.$(ii)'$. 
\end{proof}

\begin{proposition}
Soit $f:\morph{T}{S}$ un morphisme de schémas tel que $S$ est pseudo-régulier.
\begin{description}
\item[$(i)$.] Si $T$ est pseudo-régulier, le foncteur $f^*:\morph{\DMBc{S}}{\DMBc{T}}$ induit par réstriction un foncteur $f^*:\morph{\DMBcT{S}}{\DMBcT{T}}$ qui est $w$-exact.
\item[$(ii)$.] Si $f$ est lisse, le foncteur $f^!:\morph{\DMBc{S}}{\DMBc{T}}$ induit par réstriction un foncteur $f^!:\morph{\DMBcT{S}}{\DMBcT{T}}$ qui est $w$-exact.
\item[$(iii)$.] Si $f$ est un morphisme projectif entre schémas réguliers, le foncteur $f^!:\morph{\DMBc{S}}{\DMBc{T}}$ induit par réstriction un foncteur $f^!:\morph{\DMBcT{S}}{\DMBcT{T}}$ qui est $w$-exacte.
\item[$(iv)$.] Si $S$ est lisse au dessus de $B$, le foncteur de dualité locale $\D_S$ induit une équivalence de catégories entre $\DMBcT{S}_{\tate\leqslant n}^\opp$ et $\DMBcT{S}_{\tate\geqslant -n}$ quelque soit $n\in \Z$.
\end{description}
\end{proposition}

\begin{proof}$ $
\begin{description}
\item[$(i)$.] Puisque le foncteur $f^*$ envoie les objets de $\T_S$ dans $\T_T$, il envoie $\DMBcT{S}$ dans $\DMBcT{T}$ (car il s'agit d'un foncteur de catégories triangulées).
\item[$(ii)$.] Si $f$ est lisse et $S$ pseudo-régulier alors $T$ est pseudo-régulier (\cf \ref{prop.pseudo-reg.relatif}) et on se ramène au cas $(i)$ par pureté relative $(rap. 3)$.
\item[$(iii)$.] Puisque $f$ est projectif, il s'écrit comme la composé d'une immersion fermée entre schémas réguliers et d'un morphisme lisse également entre schémas régulier. Pour la partie lisse on est ramené au cas $(ii)$ ; pour la partie 'immersion fermée' on se ramène au cas $(i)$ par pureté absolue $(rap. 3)$.
\item[$(iv)$.] Puisque $\D_S(\Un_S(x)[2x])=\D_S\Un_S(-x)[-2x]=\Un_S(s-x)[2(s-x)]$ (\cf \ref{rm.dual}) alors $\D_S$ respecte le c{\oe}ur de $\tate/\DMBcT{S}$.
\end{description}
\end{proof}

\subsection{Caractéristique d'Euler motivique.}

\begin{definition}
Soit $f:\morph{X}{S}$ un morphisme de schémas. On note
$$\W_S(X):=\W_S(f_!f^*\Un_S),\qquad\Wc_S(X):=\W_S(f_*f^!\Un_S).$$
\end{definition}

On fixe un schéma $S$. Tous les schémas considérés sont ceux de $\Sch/S$

\begin{theoreme}\label{equi.GS95.thm2}$ $
\begin{description}
\item[$(i)$.] Si $X$ est un schéma de dimension $d$, 
$$\W_S(X)\in \htt^{[\![0,d]\!]}(\Chow{S}).$$
\item[$(ii)$.] Tout morphisme propre $p:\morph{X}{Y}$ induit un morphisme de complexes (à homotopie près)
$$p^\W:\morph{\W_S(Y)}{\W_S(X)}.$$
Si $p$ et $q$ sont des morphismes propres composables alors $(qp)^\W=p^\W q^\W$.

Toute immersion ouverte $j:\morph{X}{Y}$ induit un morphisme de complexes (à homotopie près)
$$j_\W:\morph{\W_S(X)}{\W_S(Y)}.$$
Si $j$ et $k$ sont des immersions ouvertes composables alors $(kj)_\W=k_\W j_\W$.
\item[$(iii)$.] Soit $j:\immop{U}{X}$ une immersion ouverte de fermé complémentaire $i:\immcl{Z}{X}$. Alors on a le triangle distingué 
de $\htt^b(\Chow{S})$,
$$\xymatrix@C=2cm{\W_S(U)\ar[r]^{j_\W}&\W_S(X)\ar[r]^{i^\W}&\W_S(Z)\ar^{+1}[r]&}.$$
\item[$(iv)$.]
\setlength{\columnseprule}{1pt}
\begin{multicols}{2}
Soit
$$\xymatrix{W\immouv[r]^{j'}\immouv[d]_{k'}\ar@{}[rd]|\square&Y\immouv[d]^k\\
U\immouv[r]_j&X}$$
un diagramme Nisnevich-distingué tel que $k$ soit une immersion ouverte.

\noindent(resp.
$$\xymatrix{T\immfer[r]^{i'}\ar[d]_{p'}\ar@{}[rd]|\square&Y\ar[d]^p\\
Z\immfer[r]_i&X}$$
un diagramme propre $cdh$-distingué).
\end{multicols}
Alors on a le triangle distingué de $\htt^b(\Chow{S})$
$$\xymatrix@C=2cm@R=0cm{&\W_S(W)\ar[r]^-{\left(^{\; j'_\W}_{-k'_\W}\right)}&\W_S(U)\oplus\W_S(V)\ar[r]^-{(k_\W, j_\W)}&\W_S(X)\ar[r]^{+1}&\\
\text{(resp.}&\W_S(X)\ar[r]_-{\left(^{\; i^\W}_{-p^\W}\right)}&\W_S(Z)\oplus\W_S(Y)\ar[r]_-{(p'^\W, i'^\W)}&\W_S(T)\ar[r]^{+1}&).}$$
\end{description}
\end{theoreme}

\begin{proof}$ $
\begin{description}
\item[$(i)$.] Cela suit de \ref{thm.mesure.motifs} et \ref{prop.Bo.3.3.1.IV}.
\item[$(ii)$.] Notons $x:\morph{X}{S}$ et $y:\morph{Y}{S}$ les morphismes structuraux de $X$ et $Y$ : $yp=x$. On considère le morphisme d'adjonction $\morph{\Un_Y}{p_*p^*\Un_Y}$ qui correspond, puisque $p$ est propre, à $\morph{\Un_Y}{p_!\Un_X}$. On compose ce dernier morphisme par $y_!$ pour obtenir
$\morph{y_!\Un_Y}{x_!\Un_X}$. Finalement
$$p^\W:=\W_S(\morph{y_!\Un_Y}{x_!\Un_X}).$$
Soit $q:\morph{Y}{Z}$ un morphisme propre. Alors l'adjonction $\morph{\Un_Z}{(qp)_*(qp)^*\Un_Z}$, s'identifie à la composée
$\morphp{\Un_Z}{q_*q^*\Un_Z}{q_*p_*p^*q^*\Un_Z}$, ce qui justifie que $(qp)^\W=p^\W q^\W$.

Dualement, l'adjonction $\morph{j_\sharp j^*\Un_Y}{\Un_Y}$ composée avec $y_!$ donne $\morph{x_!\Un_X}{y_!\Un_Y}$. Finalement
$$j_\W:=\W_S(\morph{x_!\Un_X}{y_!\Un_Y}).$$
Le même raisonnement que précédemment prouve que $(kj)_\W=k_\W j_\W$.
\item[$(iii)$.] Notons respectivement $u$, $z$ et $x$ les morphismes structuraux de $U$, $Z$ et $X$ ; on a $xj=u$ et $xi=z$. On part du triangle de localisation 
$\tridis{j_!\Un_U}{\Un_X}{i_!\Un_Z}$
que l'on compose par $x_!$ pour obtenir le triangle
$$\tridis{u_!\Un_U}{x_!\Un_X}{z_!\Un_Z}.$$
On applique le foncteur $\W_S$ pour conclure.
\item[$(iv)$.] On applique le triangle de $(rap. 2)$ :
$$\tridis{w_!\Un_W}{u_!\Un_U\oplus v_!\Un_V}{x_!\Un_X},$$
où $w$, $u$, $v$ et $x$ représentent les morphismes structuraux de $W$, $U$, $V$ et $X$. On conclut en appliquant le foncteur $\W_S$. 
Respectivement avec 
$$\tridis{x_!\Un_X}{z_!\Un_Z\oplus y_!\Un_Y}{t_!\Un_T}.$$
\end{description}
\end{proof}

\begin{corollaire}[Caractéristique d'Euler motivique]\label{cor.car.Euler}
La transformation
$$\xymatrix@R=0cm{\chi_S : \Ob(\Sch/S)\ar[r]&\KK_0(\Chow{S})\\
X\ar@{|->}[r]&\dpl{\sum_{k}(-1)^k[\W_S(X)^k]}}$$
où $[M]$ désigne la classe de $M\in \Chow{S}$ dans le groupe de Grothendieck, est bien définie et vérifie
\begin{description}
\item[$(i)$.] Si $f:\morph{X}{S}$ est propre à domaine régulier alors $\chi_S(X)=[f_!\Un_X]$.
\item[$(ii)$.] Si $U$ est un ouvert de $X$ de fermé complémentaire $Z$, on a $\chi_S(X)=\chi_S(U)+\chi_S(Z)$.
\item[$(ii)'$.] Si $Z$ est un sous-schéma fermé de $X$, on a $\chi_S(X-Z)=\chi_S(X)-\chi_S(Z)$.
\item[$(iii)$.] $ $
\setlength{\columnseprule}{1pt}
\begin{multicols}{2}
Soit
$$\xymatrix{Y\immouv[r]^{j'}\immouv[d]_{k'}\ar@{}[rd]|\square&B\immouv[d]^k\\
A\immouv[r]_j&X}$$
un diagramme Nisnevich-distingué tel que $k$ soit une immersion ouverte.

\noindent(resp.
$$\xymatrix{Y\immfer[r]^{i'}\ar[d]_{p'}\ar@{}[rd]|\square&B\ar[d]^p\\
A\immfer[r]_i&X}$$
un diagramme propre $cdh$-distingué).
\end{multicols}
alors
$$\chi_S(X)+\chi_S(Y)=\chi_S(A)+\chi_S(B).$$
\end{description}
\end{corollaire}

\begin{proof}
La transformation 
$$\xymatrix@R=0cm{\chi_S : \Ob(\Sch/S)\ar[r]&\KK_0(\htt^b(\Chow{S}))\\
X\ar@{|->}[r]&[\W_S(X)]}$$
est clairement bien définie. D'aprés \cite[lm. 3]{GS95} on a un isomorphisme de groupe $\KK_0(\htt^b(\Chow{S}))\simeq\KK_0(\Chow{S})$ qui permet l'identification de l'énoncé. 
La condition $(i)$ vient du fait que si $f$ est propre à domaine régulier $f_!\Un_X\in \Chow{S}$. La condition $(ii)$ suit de \ref{equi.GS95.thm2}.$(iii)$, la condition $(ii)'$ n'est qu'une reécriture de $(ii)$. Enfin $(iii)$ suit de \ref{equi.GS95.thm2}.$(iv)$.
\end{proof}

\begin{remarque}
Pour $S=\Spec(k)$, où $k$ désigne un corps de caractèristique $0$, le théorème \ref{equi.GS95.thm2} coïncide avec \cite[thm. 2.$(i)$, $(ii)$ ,$(iii)$ ,$(iv)$]{GS95} tandis que le corollaire \ref{cor.car.Euler} correspond à \cite[cor. 5.13]{GS09}.
\end{remarque}

\begin{theoreme}\label{equi.GS95.thm2.dual}$ $
\begin{description}
\item[$(i)$.] Si $S$ est un schéma pseudo-régulier, 
$$\Wc_S(X)\in \htt^{\leqslant 0}(\Chow{S})\cap\htt^{b}(\Chow{S}).$$
Si $S$ est régulier et lisse au dessus de $B$ et $d=\Dim(X)$,
$$\Wc_S(X)\in \htt^{[\![-d,0]\!]}(\Chow{S}).$$
\item[$(ii)$.] Tout morphisme propre $p:\morph{X}{Y}$ induit un morphisme de complexes (à homotopie près)
$$p_\Wc:\morph{\Wc_S(X)}{\Wc_S(Y)}.$$
Si $p$ et $q$ sont des morphismes propres composables alors $(qp)_\Wc=q_\Wc p_\Wc$.

Toute immersion ouverte $j:\morph{X}{Y}$ induit un morphisme de complexes (à homotopie près)
$$j^\Wc:\morph{\Wc_S(Y)}{\Wc_S(X)}.$$
Si $j$ et $k$ sont des immersions ouvertes composables alors $(kj)^\Wc=j^\Wc k^\Wc $.
\item[$(iii)$.] Soit $j:\immop{U}{X}$ une immersion ouverte de fermé complémentaire $i:\immcl{Z}{X}$. Alors on a le triangle distingué 
de $\htt^b(\Chow{S})$,
$$\xymatrix@C=2cm{\Wc_S(Z)\ar[r]^{i_\Wc}&\Wc_S(X)\ar[r]^{j^\Wc}&\Wc_S(U)\ar^{+1}[r]&}$$
\item[$(iv)$.]$ $
\setlength{\columnseprule}{1pt}
\begin{multicols}{2}
Soit
$$\xymatrix{W\immouv[r]^{j'}\immouv[d]_{k'}\ar@{}[rd]|\square&Y\immouv[d]^k\\
U\immouv[r]_j&X}$$
un diagramme Nisnevich-distingué tel que $k$ soit une immersion ouverte.

\noindent(resp.
$$\xymatrix{T\immfer[r]^{i'}\ar[d]_{p'}\ar@{}[rd]|\square&Y\ar[d]^p\\
Z\immfer[r]_i&X}$$
un diagramme propre $cdh$-distingué).
\end{multicols}
Alors on a le triangle distingué de $\htt^b(\Chow{S})$
$$\xymatrix@C=2cm@R=0cm{&\Wc_S(X)\ar[r]^-{\left(^{j^\Wc}_{k^\Wc}\right)}&\Wc_S(U)\oplus\Wc_S(V)\ar[r]^-{(-k'^\Wc, j'^\W)}&\W_S(W)\ar[r]^{+1}&\\
\text{(resp.}&\Wc_S(T)\ar[r]_-{\left(^{p'_\Wc}_{i'_\Wc}\right)}&\Wc_S(Z)\oplus\Wc_S(Y)\ar[r]_-{(i_\Wc, -p_\Wc)}&\Wc_S(X)\ar[r]^{+1}&).}$$
\end{description}
\end{theoreme}

\begin{proof} On raisonne comme dans la preuve de \ref{equi.GS95.thm2}. De manière sibylline :
\begin{description}
\item[$(i)$.] \cf \ref{thm.mesure.motifs.2}, \ref{prop.Bo.3.3.1.IV}, \ref{thm.mesure.motifs.3}.
\item[$(ii)$.] On pose $p_\Wc:=\W_Sy_*\left(\morph{p_*p^!(y^!\Un_S)}{(y^!\Un_S)}\right)$ et $j^\Wc:=\W_Sy_*\left(\morph{y^!\Un_S}{j_* j^*(y^!\Un_S)}\right)$.
\item[$(iii)$.] 
$\W_Sx_*\left(\tridis{i_*i^!(x^!\Un_S)}{(x^!\Un_S)}{j_*j^!(x^!\Un_S)}\right)$.
\item[$(iv)$.] $\W_Sx_*\left(\tridis{x^!\Un_S}{k_*k^!(x^!\Un_S)\oplus j_*j^!(x^!\Un_S)}{(kj')_*(kj')^!(x^!\Un_S)}\right)$,\\
$\W_Sx_*\left(\tridis{(pi')_*(pi')^!(x^!\Un_S)}{p_*p^!(x^!\Un_S)\oplus i_*i^!(x^!\Un_S)}{x^!\Un_S}\right)$.
\end{description}
\end{proof}

\begin{corollaire}[Caractéristique d'Euler motivique (version duale)]\label{car.Euler.dual}
La transformation
$$\xymatrix@R=0cm{\chic_S : \Ob(\Sch/S)\ar[r]&\KK_0(\Chow{S})\\
X\ar@{|->}[r]&\dpl{\sum_{k}(-1)^k[\Wc_S(X)^k]}}$$
où $[M]$ désigne la classe de $M\in \Chow{S}$ dans le groupe de Grothendieck, est bien définie et vérifie
\begin{description}
\item[$(i)$.] Si $f:\morph{X}{S}$ est un morphisme propre entre schémas réguliers, $\chic_S(X)=[f_*f^!\Un_S]$.
\item[$(ii)$.] Si $U$ est un ouvert de $X$ de fermé complémentaire $Z$, on a $\chic_S(X)=\chic_S(U)+\chic_S(Z)$.
\item[$(ii)'$.] Si $Z$ est un sous-schéma fermé de $X$, on a $\chic_S(X-Z)=\chic_S(X)-\chic_S(Z)$.
\item[$(iii)$.]
\setlength{\columnseprule}{1pt}
\begin{multicols}{2}
Soit
$$\xymatrix{Y\immouv[r]^{j'}\immouv[d]_{k'}\ar@{}[rd]|\square&B\immouv[d]^k\\
A\immouv[r]_j&X}$$
un diagramme Nisnevich-distingué tel que $k$ soit une immersion ouverte.

\noindent(resp.
$$\xymatrix{Y\immfer[r]^{i'}\ar[d]_{p'}\ar@{}[rd]|\square&B\ar[d]^p\\
A\immfer[r]_i&X}$$
un diagramme propre $cdh$-distingué).
\end{multicols}
alors
$$\chic_S(X)+\chic_S(Y)=\chic_S(A)+\chic_S(B).$$
\end{description}
\end{corollaire}

\begin{proof}
On raisonne comme en \ref{cor.car.Euler}.
\end{proof}

\begin{remarque}
La caractéristique d'Euler motivique version duale fut introduite par Bondarko dans \cite{Bo2} dans lequel, il démontre la partie $(ii)'$ du précédent corollaire (\cf \cite[prop. 3.2.1]{Bo2}).
\end{remarque}

\section{Remerciements.}
\addcontentsline{toc}{section}{Remerciements.}
Je remercie Frédéric Déglise, pour avoir pris le temps de répondre à toutes mes questions.
Je tiens particulièrement à remercier Bradley Drew pour le temps et l'intérêt qu'il a porté à ce travail ; notament pour m'avoir donné une preuve simple du lemme \ref{lm.Brad}.

\addcontentsline{toc}{section}{Bibliographie}
\bibliographystyle{amsalpha}
\bibliography{Bibliographie}

\end{document}